\documentclass[12pt,reqno]{amsart}
\usepackage{amssymb}

\usepackage{mathrsfs}

\numberwithin{equation}{section}  
\theoremstyle{plain}
\newtheorem{thm}{Theorem}[section]

\newtheorem{cor}[thm]{Corollary}
\newtheorem{lem}{Lemma}[section]

\theoremstyle{definition}

\newtheorem{example}{Example}[section]

\theoremstyle{remark}

\def\R{{\mathbb R}}
\def\C{{\mathbb C}}
\def\Rn{{{\mathbb R}^n}}

\def\Rnx{{\R^n_x}}

\def\Sph{{\mathbb S}}
\def\FT{{\mathcal F}}
\def\L2tx{{L^2(\R_t\times\R^n_x)}}
\def\Lx{{L^2(\R^n_x)}}
\def\p#1{{\left({#1}\right)}}
\def\b#1{{\left\{{#1}\right\}}}

\def\n#1{{\left\|{#1}\right\|}}
\def\abs#1{{\left|{#1}\right|}}
\def\jp#1{{\left\langle{#1}\right\rangle}}

\def\supp{\operatorname{supp}}
\def\tr{\operatorname{Tr}}
\def\rank{\operatorname{rank}}
\def\va{\varphi}
\def\ka{\kappa}

\title{Smoothing estimates for non-dispersive equations}

\author[]{Michael Ruzhansky and Mitsuru Sugimoto}
\address{
  Michael Ruzhansky:
  \endgraf
  Department of Mathematics
  \endgraf
  Imperial College London
  \endgraf
  180 Queen's Gate, London SW7 2AZ, UK
  \endgraf
  {\it E-mail address} {\rm m.ruzhansky@imperial.ac.uk}
  \endgraf
  \medskip
  Mitsuru Sugimoto:
  \endgraf
  Graduate School of Mathematics
  \endgraf
  Nagoya University
  \endgraf
  Furocho, Chikusa-ku, Nagoya 464-8602, Japan
  \endgraf
  {\it E-mail address} {\rm sugimoto@math.nagoya-u.ac.jp}
  }

\thanks{The first author was supported in parts by 
the EPSRC Grant EP/K039407/1
and by the
Leverhulme Grant RPG-2014-02.
No new data was collected or generated during the course of the research.}
\thanks{The second author was supported in parts by 
the JSPS KAKENHI 26287022 and 26610021.}


\date{\today}

\begin{document}

\begin{abstract}
This paper 
describes an approach to global smoothing problems
for non-dispersive equations
based on ideas of comparison principle and canonical transformation
established in authors' previous paper \cite{RS4},
where dispersive equations were treated.
For operators $a(D_x)$ of order $m$
satisfying the dispersiveness condition $\nabla a(\xi)\neq0$ for $\xi\not=0$,
the global smoothing estimate
\begin{equation*}
\n{\jp{x}^{-s}|D_x|^{(m-1)/2}e^{ita(D_x)}
\varphi(x)}_{L^2\p{\R_t\times\R^n_x}}
\leq C\n{\varphi}_{L^2\p{\R^n_x}} 
\quad {\rm(}s>1/2{\rm)}
\end{equation*}
is well-known,
while it is also known to fail for non-dispersive operators.
For the case when the dispersiveness breaks, we suggest the estimate 
in the form
\begin{equation*}
\n{\jp{x}^{-s}|\nabla a(D_x)|^{1/2}
e^{it a(D_x)}\varphi(x)}_{L^2\p{\R_t\times\R^n_x}}
\leq C\n{\varphi}_{L^2\p{\R^n_x}} \quad{\rm(}s>1/2{\rm)}
\end{equation*}
which is equivalent to the usual estimate in the dispersive case and
is also invariant under canonical transformations
for the operator $a(D_x)$.
We show that this estimate and its variants do continue to hold for a variety of non-dispersive
operators $a(D_x)$, where $\nabla a(\xi)$ may become zero on some set.
Moreover, other types of such estimates, and the case of time-dependent equations are also discussed.
\end{abstract}

\maketitle


\section{Introduction}
Various kinds of smoothing estimates
for the solutions $u(t,x)=e^{ita(D_x)}\varphi(x)$ to 
equations of general form
\begin{equation}\label{EQ:Disp}
\left\{
\begin{aligned}
\p{i\partial_t+a(D_x)}\,u(t,x)&=0\quad\text{in $\R_t\times\R^n_x$},\\
u(0,x)&=\varphi(x)\quad\text{in $\R^n_x$},
\end{aligned}
\right.
\end{equation}
where $a(D_x)$ is the corresponding Fourier multiplier to a real-valued
function $a(\xi)$,
have been extensively
studied under the ellipticity ($a(\xi)\not=0$)
or the dispersiveness ($\nabla a(\xi)\not=0$) conditions,
exclusive of the origin $\xi=0$ (i.e. for $\xi\not=0$) when $a(\xi)$ is a homogeneous function.
Such conditions include Schr\"odinger equation as a special case
($a(\xi)=|\xi|^2$).
Since Kato's local gain of one derivative for the linearised KdV in \cite{Ka1}, and 
the independent works
by Ben-Artzi and Devinatz \cite{BD2}, Constantin and Saut \cite{CS},
Sj\"olin \cite{Sj} and Vega \cite{V}, 
the local, and then global smoothing estimates,
together with their application to non-linear problem
have been intensively investigated in a series of papers such
as \cite{BK}, \cite{BN}, \cite{Ch1}, \cite{Ho1}, \cite{Ho2},
\cite{KPV1}, \cite{KPV2},
\cite{KPV3}, \cite{KPV4}, \cite{KPV5}, \cite{KPV6}, \cite{KY}, \cite{LP}
\cite{RS1}, \cite{RS2}, \cite{RS3}, \cite{RS4}, \cite{Si}
\cite{Su1}, \cite{Su2}, \cite{W}, \cite{Wa1} \cite{Wa2}, to mention a few.
\par
\medskip
Among them, a comprehensive analysis is presented
in our previous paper \cite{RS4} for global
smoothing estimates by using two useful methods, that is, 
{\it canonical transformations} and the the {\it comparison principle}.
Canonical transformations are a tool to transform one equation to another
at the estimate level, and the comparison principle is a tool to
relate differential estimates for solutions to different equations.
These two methods work very effectively under the dispersiveness
conditions to induce a number of new or refined global smoothing estimates,
as well as many equivalences between them.
Using these methods, the proofs of smoothing estimates are also considerably simplified.

\medskip
The objective of this paper is to continue the investigation
by the same approach in the case when the dispersiveness breaks down.
We will conjecture what we may call an `invariant estimate' extending the smoothing
estimates to the non-dispersive case. Such an estimate yields the known smoothing
estimates in dispersive cases, it is invariant under canonical transforms of the problem,
and we will show its validity for a number of non-dispersive evolution equations of
several different types.

\medskip
The most typical example of a global smoothing estimate is of the form
\begin{equation}\label{EQ:main1-0}
\n{\jp{x}^{-s}|D_x|^{(m-1)/2}e^{ita(D_x)}
\varphi(x)}_{L^2\p{\R_t\times\R^n_x}}
\leq C\n{\varphi}_{L^2\p{\R^n_x}}  \qquad (s>1/2),
\end{equation}
where $m$ denotes the order of the operator $a(D_x)$,
and this estimate, together with other similar kind of
global smoothing estimates, has been already justified
under appropriate dispersiveness assumptions
(see Section \ref{subsection2.1}).
Throughout this paper we use the standard notation
\[
\jp{x}=\p{1+|x|^2}^{1/2} \quad\textrm{and}\quad
\jp{D_x}=\p{1-\Delta_x}^{1/2}.
\]
Note that the $L^2$-norm of the solution is always the same
as that of the Cauchy data $\varphi$ for any fixed time $t$, but
estimate \eqref{EQ:main1-0} means that the extra gain of
regularity of order $(m-1)/2$ in the spacial variable $x$ can
be observed if we integrate the solution $e^{it a(D_x)}\varphi(x)$
to equation \eqref{EQ:Disp} in the time variable $t$.
One interesting conclusion in \cite{RS4} is that
our method allowed us to carry out a global
microlocal reduction of estimate \eqref{EQ:main1-0}
to the translation invariance of the Lebesgue measure.

\medskip
On the other hand, despite their natural appearance
in many problems, quite limited results are available for
non-dispersive equations while the dispersiveness condition was shown
to be necessary for most common types of global smoothing
estimates (see \cite{Ho2}).
To give an example, coupled dispersive equations are of high importance in 
applications while only limited analysis is available.
Let $v(t,x)$ and $w(t,x)$
solve the following coupled system of Schr\"odinger equations:
\begin{equation}\label{eq2}
\left\{
\begin{aligned} 
i\partial_t v & =\Delta_x v+b(D_x)w, \\
i\partial_t w & =\Delta_x w+c(D_x)v,  \\
v(0,x) & =v_0(x), w(0,x)=w_0(x). 
\end{aligned}
\right.
\end{equation}
This is the simplest example of Schr\"odinger equations coupled
through linearised operators $b(D_x), c(D_x)$. 
Such equations appear
in many areas in physics. For example, this is a model of wave
packets with two modes (in the presence of resonances),
see Tan and Boyd \cite{TB}.
In fibre optics they appear to describe certain 
types of a pair 
of coupled modulated wave-trains (see e.g.
Manganaro and Parker \cite{MP}).
They also describe the field of optical solitons in fibres
(see Zen and Elim \cite{ZE}) as well as Kerr dispersion and stimulated
Raman scattering for ultrashort pulses transmitted through fibres.
In these cases the linearised operators $b$ and $c$ 
would be of
zero order. In models of optical pulse propagation of birefringent
fibres and in wavelength-division-multiplexed systems they are
of the first order (see Pelinovsky and Yang \cite{PY}). 
They may be of higher orders as well, for example
in models of optical solitons with higher order effects 
(see Nakkeeran \cite{Na}).

\medskip
Suppose now that we are in the simplest 
situation when system 
\eqref{eq2} can be diagonalised. Its eigenvalues are 
$a_\pm(\xi)=-|\xi|^2\pm\sqrt{b(\xi)c(\xi)}$ and the system uncouples
into scalar equations of type \eqref{EQ:Disp}
with operators $a(D_x)=a_\pm(D_x)$.
Since the structure of operators $b(D_x), c(D_x)$ 
may be quite involved,
this motivates the study of scalar equations 
\eqref{EQ:Disp} 
with 
operators $a(D_x)$ of rather general form. Not only the presence
of lower order terms is important in time global problems, 
the principal part may be
rather general since we may have $\nabla a_\pm=0$ 
at some points.

\medskip
Let us also briefly mention another concrete example.
Equations of the third order often appear in applications to
water wave equations. For example, the Shrira
equation \cite{Sh} describing the propagation of a 
three-dimensional packet of weakly nonlinear internal gravity 
waves leads to third order polynomials in two dimensions.
The same types of third order polynomials in two variables also
appear in the Dysthe equation as well as in the Hogan equation,
both describing the behaviour of deep water waves in 2-dimensions.
Strichartz estimates for the corresponding solutions have 
been analysed by e.g. Ghidaglia and Saut \cite{GS} and by
Ben-Artzi, Koch and Saut \cite{BKS} by reducing the equations
to pointwise estimates for operators in normal forms.
In general, by linear changes of variables, polynomials $a(\xi_1,\xi_2)$ of
order $3$
are reduced to one of the following normal forms:
\begin{align}\nonumber
&
\xi_1^3,\quad  \xi_1^3+\xi_2^3,\quad \xi_1^3-\xi_1\xi_2^2,
\\ \label{EQ:examples2}
&
\xi_1^3+\xi_2^2,\quad \xi_1\xi_2^2, \quad\xi_1\xi_2^2+\xi_1^2,
\\
&
\xi_1^3+\xi_1\xi_2,\quad \xi_1^3+\xi_2^3+\xi_1\xi_2,
\quad \xi_1^3-3\xi_1\xi_2^2+\xi_1^2+\xi_2^2,
\nonumber
\end{align}
modulo polynomials of order one. 
Strichartz estimates have been obtained for operators having their
symbols in this list except
for the cases $a(\xi_1,\xi_2)=\xi_1^3$ and $\xi_1\xi_2^2$,
see Ben-Artzi, Koch and Saut \cite{BKS}.
Some of normal forms listed here
satisfy dispersiveness assumptions and can be discussed
by existing results listed in Section \ref{subsection2.1}.
Indeed, $a(\xi_1,\xi_2)=\xi_1^3+\xi_2^3$ and
$\xi_1^3-\xi_1\xi_2^2$ are homogeneous and satisfy
the dispersiveness $\nabla a(\xi_1,\xi_2)\neq0$ except for the origin
(which is assumption (H) in Section \ref{subsection2.1}).
However the dispersiveness assumption is sensitive to the perturbation
by polynomials of order one. 
For example, $a(\xi_1,\xi_2)=\xi_1^3+\xi_2^3+\xi_1$ still satisfies
the dispersiveness $\nabla a(\xi_1,\xi_2)\neq0$ everywhere
(see assumption (L) in Section \ref{subsection2.1}),
but it breaks for $a(\xi_1,\xi_2)=\xi_1^3+\xi_2^3-\xi_1$.
Furthermore the other normal forms
do not satisfy neither of these assumptions, with the corresponding
equation losing dispersiveness at some points
(see Examples \ref{nondispersive1} and \ref{ex:isolated-critical}).

\medskip

We now turn to describing an estimate which holds in such cases even 
when the dispersiveness fails at some points in the phase space.
In this paper, based on the methods of comparison principle and 
canonical transforms, we develop several approaches to getting 
smoothing estimates in such non-dispersive cases.
Since standard global smoothing estimates are known to fail in
non-dispersive cases by \cite{Ho2}, 
we will suggest an invariant form of global smoothing estimates
instead, for the analysis, and call them {\it invariant estimates},
which we expect to continue to hold even in non-dispersive cases.
As an example of it,
estimate \eqref{EQ:main1-0}
may be rewritten in the form
\begin{equation}\label{EQ:main-invariant}
\n{\jp{x}^{-s}|\nabla a(D_x)|^{1/2}
e^{it a(D_x)}\varphi(x)}_{L^2\p{\R_t\times\R^n_x}}
\leq C\n{\varphi}_{L^2\p{\R^n_x}} \qquad (s>1/2).
\end{equation}
Indeed the normal form $a(\xi_1,\xi_2)=\xi_1^3$ listed in \eqref {EQ:examples2}
is known to satisfy this estimate (see estimate \eqref{model:1}
in Corollary \ref{Th:typeI}).
Other types of invariant estimates are also suggested in Section
\ref{subsection2.2}.
Such estimate has a number of advantages:
\begin{itemize}
\item in the dispersive case it is equivalent to the usual
estimate \eqref{EQ:main1-0};
\item it does continue to hold for a variety of non-dispersive
equations, where $\nabla a(\xi)$ may become zero on some set
and when \eqref{EQ:main1-0} fails;
\item it does take into account possible zeros of the gradient
$\nabla a(\xi)$ in the non-dispersive case,
which is also responsible for the interface
between dispersive and non-dispersive zone (e.g. how quickly
the gradient vanishes);
\item it is invariant under canonical transformations of the equation.
\end{itemize}

\medskip
We will also try to justify invariant estimates for non-dispersive equations.
The combination of the proposed two methods
(canonical transformations and the comparison principles)
has a good power again on the occasion of this analysis.
Besides the simplification of the proofs of global smoothing estimates
for standard dispersive equations,
we have the following advantage in treating non-dispersive equations
where the dispersiveness condition $\nabla a(\xi)\neq 0$ breaks:
\begin{itemize}
\item in radially symmetric cases, we can use the comparison
principle of radially symmetric type (Theorem \ref{Th:nondisprad});
\item in polynomial cases
we can use the comparison principle of one dimensional type
(Theorem \ref{polynomial});
\item in the homogeneous case with some information on the
Hessian, we can use canonical transformation to reduce the
general case to some well-known model situations 
(Theorem \ref{th1});
\item around non-dispersive points where the Hessian is non-degenerate,
we can microlocalise and apply the canonical transformation based on
the Morse lemma (Theorem \ref{THM:isolated-critical});
\end{itemize}

\medskip
In particular, in the radially symmetric cases,
we will see that estimate \eqref{EQ:main-invariant} is
valid in a generic situation (Theorem \ref{Th:nondisprad}).
And as another remarkable result, it is also valid for
any differential operators with real constant coefficients,
including operators corresponding to normal forms listed
in \eqref {EQ:examples2} with
perturbation by polynomials of order one (Theorem \ref{polynomial}).
Some normal forms are also covered by Theorem \ref{THM:isolated-critical}.

\medskip 
In addition, we will derive estimates for equations with time
dependent coefficients. In general, the dispersive estimates
for equations with time dependent coefficients may be a
delicate problem, with decay rates heavily depending on the
oscillation in coefficients (for a survey of different
results for the wave equation with lower order terms see,
e.g. Reissig \cite{Rei}; for more general equations and systems and
the geometric analysis of the time-decay rate of their solution
see \cite{RW} or \cite{CUFRW}). However, we will show in Section
\ref{SECTION:time-dependent} that the smoothing estimates still
remain valid if we introduce an appropriate factor into the
estimate. Such estimates become a natural extension of the
invariant estimates to the time dependent setting.

\medskip 
We will explain the organisation of this paper.
In Section \ref{SECTION:invariant}, we list
typical global smoothing estimates for dispersive equations,
and then discuss their invariant form which we expect to
remain true also in non-dispersive situations.
In Section \ref{SECTION:nondisp},
we establish invariant estimates for several types of
non-dispersive equations.
The case of time--dependent coefficients will be treated
in Section \ref{SECTION:time-dependent}.
In Appendix \ref{SECTION:tools},
we review our fundamental tools, that is, the canonical transformation
and the comparison principle, which is used for
the analysis in Section \ref{SECTION:nondisp}.

\medskip 
Finally we comment on the notation used in this paper.
When we need to specify the entries of the vectors $x,\xi\in\R^n$,
we write $x=(x_1,x_2,\ldots,x_n)$, $\xi=(\xi_1,\xi_2,\ldots,\xi_n)$
without any notification.
As usual we will denote
$\nabla=(\partial_1,\ldots,\partial_n)$ where
$\partial_j=\partial_{x_j}$,
$D_x=(D_1,D_2\ldots,D_n)$ where
$D_{j}=-\sqrt{-1}\,\partial_{j}$ ($j=1,2,\ldots,n$),
and
view operators $a(D_x)$ as Fourier multipliers.
We denote the set of the positive real numbers $(0,\infty)$
by $\R_+$.
Constants denoted by letter $C$ 
in estimates are always positive and
may differ on different 
occasions, but will still be denoted by the same letter.

\section{Invariant smoothing estimates for dispersive equations}
\label{SECTION:invariant}

In this section we collect known smoothing estimates for dispersive equations
under several different assumptions.
Then we show that the invariant estimate \eqref{EQ:main-invariant} holds
in these cases and is, in fact, equivalent to the known estimates.
Thus, let us consider the solution
\[
u(t,x)=e^{ita(D_x)}\varphi(x)
\]
to the equation
\[
\left\{
\begin{aligned}
\p{i\partial_t+a(D_x)}\,u(t,x)&=0\quad\text{in $\R_t\times\R^n_x$},\\
u(0,x)&=\varphi(x)\quad\text{in $\R^n_x$},
\end{aligned}
\right.
\]
where we always assume that 
function $a(\xi)$ is real-valued.
We sometimes decompose the initial data $\varphi$ into the sum of
the {\it low frequency} part
$\varphi_{l}$ and the {\it high frequency} part $\varphi_{h}$, where
$\supp\widehat{\varphi_l}\subset\b{\xi:|\xi|< 2R}$
and $\supp\widehat{\varphi_h}\subset\b{\xi:|\xi|> R}$
with sufficiently large $R>0$.
First we review a selection of known results on global smoothing estimates
established in \cite{RS4}
when the dispersiveness assumption $\nabla a(\xi)\neq0$  for $\xi\neq 0$ is satisfied,
and then rewrite them in a form which is expected to hold even in 
non-dispersive situations.
\medspace

\subsection{Smoothing estimates for dispersive equations}
\label{subsection2.1}

First we collect known results for dispersive equations.
Let $a_m=a_m(\xi)\in C^\infty(\R^n\setminus0)$,
the {\it principal} part of $a(\xi)$, be 
a positively homogeneous function
of order $m$, that is, satisfy
$$a_m(\lambda\xi)=\lambda^m a_m(\xi) \textrm{ for all } \lambda>0 \textrm{ and } \xi\neq0.$$
First we consider the case that $a(\xi)$
has no lower order terms,
and assume that $a(\xi)$ is {\it dispersive}:
\begin{equation}\tag{{\bf H}}
a(\xi)=a_m(\xi),\qquad\nabla a_m(\xi)\neq0 \quad(\xi\in\R^n\setminus0).
\end{equation}
\par\noindent

\par
\medskip 
\par
\begin{thm}\label{M:H1}
Assume {\rm{(H)}}. Then there exists a constant $C>0$ such that the following estimates hold true.
\begin{itemize}
\item Suppose $n\geq 1$, $m>0$, and $s>1/2$.
Then we have
\begin{equation}\label{EQ:main1}
\n{\jp{x}^{-s}|D_x|^{(m-1)/2}e^{ita(D_x)}\varphi(x)}_{L^2\p{\R_t\times\R^n_x}}
\leq C\n{\varphi}_{L^2\p{\R^n_x}}.
\end{equation}
\item Suppose $m>0$ and $(m-n+1)/2<\alpha<(m-1)/2$.
Or, in the elliptic case $a(\xi)\neq0$
$(\xi\neq0)$, suppose $m>0$ and $(m-n)/2<\alpha<(m-1)/2$.
Then we have
\begin{equation}\label{EQ:main5}
\n{\abs{x}^{\alpha-m/2}|D_x|^{\alpha}e^{ita(D_x)}\varphi(x)}_
{L^2\p{\R_t\times\R^n_x}}
\leq C\n{\varphi}_{L^2\p{\R^n_x}}.
\end{equation}
\item Suppose $n-1>m>1$, but in the elliptic case $a(\xi)\neq0$
$(\xi\neq0)$ suppose $n>m>1$. Then we have
\begin{equation}\label{EQ:main3}
\n{\jp{x}^{-m/2}\jp{D_x}^{(m-1)/2}e^{ita(D_x)}\varphi(x)}_
{L^2\p{\R_t\times\R^n_x}}
\leq C\n{\varphi}_{L^2\p{\R^n_x}}.
\end{equation}
\end{itemize}
\end{thm}
\par
\medskip 
\par
In the Schr\"odinger equation case $a(\xi)=|\xi|^2$ and for $n\geq 3$,
estimate \eqref{EQ:main1} was obtained by 
Ben-Artzi and Klainerman \cite{BK}.
It follows also from a sharp local smoothing estimate by
Kenig, Ponce and Vega \cite[Theorem 4.1]{KPV1}), and also from the one by
Chihara \cite{Ch1} who treated the case $m>1$. 
For the range $m>0$ and any $n\geq 1$ it was obtained in 
\cite[Theorem 5.1]{RS4}.

Compared to \eqref{EQ:main1}, the estimate \eqref{EQ:main5}  
is scaling invariant with the homogeneous weights
$|x|^{-s}$ instead of non-homogenous ones $\jp{x}^{-s}$.
The estimate \eqref{EQ:main5}  was obtained in \cite[Theorem 5.2]{RS4},
and it is a generalisation of the result by
Kato and Yajima \cite{KY} who treated the case $a(\xi)=|\xi|^2$
with $n\geq 3$ and $0\leq \alpha<1/2$, or with $n=2$ and $0<\alpha<1/2$,
and also of the one by Sugimoto \cite{Su1} who treated elliptic $a(\xi)$ 
of order $m=2$ with $n\geq 2$ and
$1-n/2<\alpha<1/2$. 

The smoothing estimate \eqref{EQ:main3} is of yet another type replacing
$|D_x|^{(m-1)/2}$ by its non-homogeneous version
$\jp{D_x}^{(m-1)/2}$, obtained in \cite[Corollary 5.3]{RS4}.
It is a direct consequence of 
\eqref{EQ:main1} with $s=m/2$ and \eqref{EQ:main5} with
$\alpha=0$ (note also the $L^2$--boundedness of
$\jp{D_x}^{(m-1)/2}(1+|D_x|^{(m-1)/2})^{-1}$),
and it also extends the result by Kato and Yajima \cite{KY}
who treated the case $a(\xi)=|\xi|^2$ and $n\geq 3$,
the one by Walther \cite{Wa2} who treated the 
case $a(\xi)=|\xi|^m$, and the one by the authors \cite{RS1}
who treated the elliptic case with $m=2$.

\medskip
We can also consider
the case that $a(\xi)$ has lower order terms,
and assume that $a(\xi)$ is dispersive in the following sense:

\begin{equation}\tag{{\bf L}}
\begin{aligned}
&a(\xi)\in C^\infty(\R^n),\qquad \nabla a(\xi)\neq0 \quad(\xi\in\R^n),
\quad\nabla a_m(\xi)\neq0 \quad (\xi\in\R^n\setminus0), 
\\
&|\partial^\alpha\p{a(\xi)-a_m(\xi)}|\leq C_\alpha\abs{\xi}^{m-1-|\alpha|}
\quad\text{for all multi-indices $\alpha$ and all $|\xi|\geq1$}.
\end{aligned}
\end{equation}
\medskip
\par\noindent
Condition (L) may be formulated equivalently in the following way
\medskip
\begin{equation}\tag{{\bf L}}
\begin{aligned}
&a(\xi)\in C^\infty(\R^n),\qquad 
|\nabla a(\xi)|\geq C\jp{\xi}^{m-1}\quad(\xi\in\R^n)\quad
\textrm{for some}\; C>0,
\\
&|\partial^\alpha\p{a(\xi)-a_m(\xi)}|\leq C_\alpha\abs{\xi}^{m-1-|\alpha|}
\quad\text{for all multi-indices $\alpha$ and all $|\xi|\geq1$}.
\end{aligned}
\end{equation}
The last line of these assumptions simply amount to saying that
the principal part $a_m$ of $a$ is positively homogeneous
of order $m$ for $|\xi|\geq 1$.
Then we have the following estimates:
\par
\medskip 
\par
\begin{thm}\label{M:L4}
Assume {\rm (L)}.
Then there exists a constant $C>0$ such that the following estimates hold true.
\begin{itemize}
\item 
Suppose $n\geq1$, $m>0$, and $s>1/2$.
Then we have
\begin{equation}\label{EQ:main4}
\n{\jp{x}^{-s}\jp{D_x}^{(m-1)/2}e^{ita(D_x)}\varphi(x)}
_{L^2\p{\R_t\times\R^n_x}}
\leq C\n{\varphi}_{L^2\p{\R^n_x}}.
\end{equation}
\item Suppose $n\geq1$, $m\geq1$ and $s>1/2$.
Then we have
\begin{equation}\label{EQ:main9}
\n{\jp{x}^{-s}|D_x|^{(m-1)/2}e^{ita(D_x)}\varphi(x)}_{L^2\p{\R_t\times\R^n_x}}
\leq C\n{\varphi}_{L^2\p{\R^n_x}}.
\end{equation}
\end{itemize}
\end{thm}
\par
\medskip 
\par
The estimate \eqref{EQ:main4} was established in \cite[Theorem 5.4]{RS4}.
Consequently, \eqref{EQ:main9} is a straightforward consequence
of \eqref{EQ:main4} and the $L^2$-boundedness
of the Fourier multiplier $|D_x|^{(m-1)/2}\jp{D_x}^{-(m-1)/2}$
with $m\geq1$. It is an analogue of \eqref{EQ:main1} for operators $a(D_x)$
with lower order terms, and also a generalisation of \cite[Theorem 4.1]{KPV1}
who treated essentially polynomial symbols $a(\xi)$.
For \eqref{EQ:main9} in its full generality we refer to \cite[Corollary 5.5]{RS4}.

\medskip
Assumption (L) 
requires the condition
$\nabla a(\xi)\neq0$ ($\xi\in\R^n$) for the full symbol, 
besides the same
one $\nabla a_m(\xi)\neq0$ ($\xi\neq0$) for the principal term.
To discuss what happens if we do not have
the condition $\nabla a(\xi)\neq0$, we can
introduce an intermediate assumption between (H) and (L):
\medskip
\begin{equation}\tag{{\bf HL}}
\begin{aligned}
&a(\xi)=a_m(\xi)+r(\xi),\quad
\nabla a_m(\xi)\neq0 \quad (\xi\in\R^n\setminus0),
\quad r(\xi)\in C^\infty(\R^n)
 \\
&|\partial^\alpha r(\xi)|\leq C\jp{\xi}^{m-1-|\alpha|}
\quad\text{for all multi-indices $\alpha$}.
\end{aligned}
\end{equation}
Theorem \ref{M:L4} remain valid if we
replace assumption (H) by (HL) and functions $\varphi(x)$ in 
the estimates by their 
(sufficiently large) high frequency parts $\varphi_h(x)$.
However we cannot control the low frequency parts $\varphi_l(x)$,
and so have only the time local estimates on the whole, which we
now state for future use:
\par
\medskip 
\par
\begin{thm}[{\cite[Theorem 5.6]{RS4}}]\label{Th:HL}
Assume {\rm (HL)}.
Suppose $n\geq1$, $m>0$, $s>1/2$, and $T>0$.
Then we have
\[
\int^T_0\n{\jp{x}^{-s}\jp{D_x}^{(m-1)/2}e^{ia(D_x)}\varphi(x)}
^2_{L^2(\R^n_x)}\,dt
\leq
 C\n{\varphi}_{L^2(\R^n)}^2,
\]
where $C>0$ is a constant depending on $T>0$.
\end{thm}
\par
\medskip 
\par

\subsection{Invariant estimates}
\label{subsection2.2}

Let us now suggest an invariant form of time global smoothing estimates
which might remain valid also in some
areas without dispersion $\nabla a(\xi)\neq0$, where standard
smoothing estimates are known to fail (see Hoshiro \cite{Ho2}).
We can equivalently rewrite estimates above under the dispersiveness
assumption (H) or (L) in the form
\begin{equation}\label{EQ:inv-form}
\n{w(x)\zeta\p{\abs{\nabla a(D_x)}}
e^{it a(D_x)}\varphi(x)}_{L^2\p{\R_t\times\R^n_x}}
\leq C\n{\varphi}_{L^2\p{\R^n_x}} \qquad (s>1/2),
\end{equation}
where $w$ is a weight function of the form
$w(x)=|x|^{\delta}$ or $\jp{x}^{\delta}$,
and the smoothing is given by the function $\zeta$ on $\R_+$ of the form
$\zeta(\rho)=\rho^\eta$ or $\p{1+\rho^2}^{\eta/2}$,
with some $\delta,\eta\in\R$.
For example, we can rewrite estimate \eqref{EQ:main1}
of Theorem \ref{M:H1} as well as estimate \eqref{EQ:main9}
of Theorem \ref{M:L4} for the dispersive equations in the form
\begin{equation}\label{EQ:maininv}
\n{\jp{x}^{-s}|\nabla a(D_x)|^{1/2}
e^{it a(D_x)}\varphi(x)}_{L^2\p{\R_t\times\R^n_x}}
\leq C\n{\varphi}_{L^2\p{\R^n_x}}.
\end{equation}
Similarly we can rewrite estimate \eqref{EQ:main5} of Theorem \ref{M:H1}
in the form
\begin{equation}\label{EQ:maininv2}
\n{\abs{x}^{\alpha-m/2}|\nabla a(D_x)|^{\alpha/(m-1)}e^{ita(D_x)}\varphi(x)}_
{L^2\p{\R_t\times\R^n_x}}
\leq C\n{\varphi}_{L^2\p{\R^n_x}}
\qquad(m\neq1),
\end{equation}
and estimate \eqref{EQ:main3} of Theorem \ref{M:H1} (with $s=-m/2$)
as well as estimate \eqref{EQ:main4} of Theorem \ref{M:L4} 
(with $s>1/2$)
in the form
\begin{equation}\label{EQ:maininv3}
\n{\jp{x}^{-s}\jp{\nabla a(D_x)}^{1/2}e^{ita(D_x)}\varphi(x)}_
{L^2\p{\R_t\times\R^n_x}}
\leq C\n{\varphi}_{L^2\p{\R^n_x}}.
\end{equation}
Indeed, under assumption (H) we clearly have 
$|\nabla a(\xi)|\geq c|\xi|^{m-1}$, so the equivalence
between estimate \eqref{EQ:maininv} 
and estimate \eqref{EQ:main1} 
in Theorem \ref{M:H1}
follows from the fact that the Fourier multipliers
$|\nabla a(D_x)|^{1/2}|D_x|^{-(m-1)/2}$ and
$|\nabla a(D_x)|^{-1/2}|D_x|^{(m-1)/2}$ are bounded
on $L^2(\Rn)$.
Under assumption (L) the same argument works for large
$|\xi|$, while for small $|\xi|$ both
$\jp{\xi}^{(m-1)/2}$ and $|\nabla a(\xi)|^{1/2}$ are
bounded away from zero.
Thus we have the equivalence between
estimate \eqref{EQ:maininv} and estimate \eqref{EQ:main9} 
in Theorem \ref{M:L4}.
The same is true for the other equivalences.
As we will see later (Theorem \ref{Th:caninv}),
estimate \eqref{EQ:inv-form},
and hence estimates \eqref{EQ:maininv}--\eqref{EQ:maininv3} are
invariant under canonical transformations.
On account of it,
we will call estimate \eqref{EQ:inv-form}
an {\it invariant estimate},
and indeed we expect invariant estimates
\eqref{EQ:maininv}, \eqref{EQ:maininv2},
and \eqref{EQ:maininv3} to hold 
without dispersiveness assumption (H) or (L),
for $s>1/2$, $(m-n)/2<\alpha<(m-1)/2$,
and $s=-m/2$ ($n>m>1$),
respectively in ordinary settings (elliptic case for example),
where $m>0$ is the order of $a(D_x)$.
We will discuss and establish them in Section \ref{SECTION:nondisp}
in a variety of situations.

\medskip

Here is an intuitive understanding of the invariant
estimate \eqref{EQ:maininv} with $s>1/2$ by spectral argument.
Let $E(\lambda)$ be the spectral family of the self-adjoint realisation
of $a(D_x)$ on $L^2(\R^n)$, that is
\[
(E(\lambda) f,g)=\int_{a(\xi)\leq\lambda}
\widehat f(\xi)\overline{\widehat g(\xi)}\,d\xi.
\]
Then its spectral density
$$A(\lambda)=\frac{d}{d\lambda}E(\lambda)
=\frac1{2\pi i}\p{R(\lambda+i0)-R(\lambda-i0)},$$
where $R(\lambda\pm i0)$ denotes the ``limit" of the resolvent
$R(\lambda\pm i\varepsilon)$ as $\varepsilon\searrow0$,
is given by
\[
(A(\lambda) f,g)
=
\frac{d}{d\lambda}(E(\lambda) f,g)=\int_{a(\xi)=\lambda}
\widehat f(\xi)\overline{\widehat g(\xi)}\,\frac{d\xi}{|\nabla a(\xi)|}
\]
whenever $\nabla a(\xi)\not=0$.
Hence we have the identity 
\[
(A(\lambda)  |\nabla a (D_x)|^{1/2}f, |\nabla a (D_x)|^{1/2}g)
=
\int_{a(\xi)=\lambda}
\widehat f(\xi)\overline{\widehat g(\xi)}\,d\xi
\]
which continues to have a meaning even if $\nabla a(\xi)$ may vanishes
(although this is just a formal observation).
On the other hand, the right hand side of this identity with $f=g$
has the uniform estimate
\[
\int_{a(\xi)=\lambda}
\abs{\widehat f(\xi)}^2\,d\xi
\leq C\n{\jp{x}^sf}_{L^2}^2
\]
in $\lambda\in\R$ by the one-dimensional Sobolev embedding.
(It can be justified at least when $\nabla a(\xi)\neq0$.
See \cite[Lemma 1]{Ch2}.)
If once we justify them, we could have the estimate
\begin{equation}\label{supersmooth}
\sup_{\lambda\in\R}|(A(\lambda)H^*f, H^*f)|
\leq C\n{f}_{L^2}^2,
\end{equation}
where $H^*$ is the adjoint operator of $H=\jp{x}^{-s}|\nabla a(D_x)|^{1/2}$.
Estimate \eqref{supersmooth} can be regarded as
the $a(D_x)$-smooth property
\[
\sup_{\mu\notin\R}|([R(\mu)-R(\bar\mu)] H^*f, H^*f)|
\leq C\n{f}_{L^2}^2
\]
which is equivalent to invariant estimate \eqref{EQ:maininv}
(see \cite[Theorem 5.1]{Ka1}).
Or we may proceed as in \cite{BD2} or \cite{BK} to obtain the same conclusion.
In fact, we have
\begin{align*}
\p{He^{ita(D_x)}\varphi(x),\,\psi(t,x)}_{L^2(\R_t\times\R^n_x)}
&=\int^\infty_{-\infty}\int_{-\infty}^\infty e^{it\lambda}
\p{A(\lambda)\varphi,\,H^*\psi(t,\cdot)}\,d\lambda dt
\\
&=\int^\infty_{-\infty}
\p{A(\lambda)\varphi,\,H^*\widetilde\psi(\lambda,\cdot)}\,d\lambda,
\end{align*}
where $\widetilde \psi(\lambda,x)=\int^\infty_{-\infty}e^{-it\lambda}\psi(t,x)\,dt$, and
\begin{align*}
\abs{\p{A(\lambda)\varphi,\,H^*\widetilde\psi(\lambda,\cdot)}}
&\leq \p{A(\lambda)\varphi,\varphi}^{1/2}
\p{A(\lambda)H^*\widetilde\psi(\lambda,\cdot),\,
H^*\widetilde\psi(\lambda,\cdot)}^{1/2}
\\
&\leq C \p{A(\lambda)\varphi,\varphi}^{1/2}
\n{\widetilde\psi(\lambda,\cdot)}_{L^2}
\end{align*}
by estimate \eqref{supersmooth}.
Hence by Schwartz inequality,
Plancherel's theorem, and the fact
$\int^\infty_{-\infty}\p{A(\lambda)\varphi,\varphi}\,d\lambda=\n{\varphi}_{L^2(\R^n_x)}^2$, we have the estimate
\[
\abs{
\p{He^{ita(D_x)}\varphi(x),\,\psi(t,x)}_{L^2(\R_t\times\R^n_x)}}
\leq C\n{\varphi}_{L^2(\R^n_x)}\n{\psi(t,x)}_{L^2(\R_t\times\R^n_x)}
\]
which is again equivalent to invariant estimate \eqref{EQ:maininv}.

\medskip
Let us now show that the invariant
estimate \eqref{EQ:maininv} with $s>1/2$ 
is also a refinement of
another known estimate for non-dispersive equations,
namely of the smoothing estimate obtained by Hoshiro in \cite{Ho1}.
If operator $a(D_x)$ has real-valued 
symbol $a=a(\xi)\in C^1(\Rn)$ 
which is positively homogeneous 
of order $m\geq 1$ and no dispersiveness assumption 
is made, Hoshiro \cite{Ho1} showed the estimate
\begin{equation}\label{EQ:Hoshiro1}
\n{\jp{x}^{-s}\jp{D_x}^{-s}|a(D)|^{1/2} e^{ita(D_x)}
\va(x)}_\L2tx\leq
C\n{\varphi}_\Lx \quad {\rm (}s>1/2{\rm)}.
\end{equation}
But once we prove \eqref{EQ:maininv} with $s>1/2$,
we can have a better estimate
$$\n{\jp{x}^{-s}\jp{D_x}^{-1/2}|a(D)|^{1/2} e^{ita(D_x)}
\va(x)}_\L2tx\leq
C\n{\varphi}_\Lx \quad {\rm (}s>1/2{\rm)}$$
with respect to the number of derivatives.
In fact, using the Euler's identity 
$$m a(\xi)=\xi\cdot\nabla a(\xi),$$
we see that this estimate trivially follows from
$$\n{\jp{x}^{-s}\jp{D_x}^{-1/2}|D_x|^{1/2}
|\nabla a(D)|^{1/2} e^{ita(D_x)}\va(x)}_\L2tx\leq
C\n{\varphi}_\Lx \quad {\rm (}s>1/2{\rm)},$$
which in turn follows from \eqref{EQ:maininv} with $s>1/2$
because the Fourier multiplier operator $\jp{D_x}^{-1/2}|D_x|^{1/2}$ is
$L^2(\Rn)$-bounded. In fact, estimate 
\eqref{EQ:Hoshiro1} holds only because of the homogeneity
of $a$, since in this case by Euler's identity
zeros of $a$ contain
zeros of $\nabla a$. In general, estimate
\eqref{EQ:Hoshiro1} cuts off too much, and therefore
does not reflect the nature of the problem for non-homogeneous
symbols, as \eqref{EQ:maininv} still does.

\medskip
In terms of invariant estimates,
we can also give another explanation to the reason
why we do not have time-global estimate in Theorem \ref{Th:HL}.
The problem is that the symbol of
the smoothing operator $\jp{D_x}^{(m-1)/2}$ does not vanish
where the symbol of $\nabla a(D_x)$ vanishes, as should
be anticipated by the invariant estimate \eqref{EQ:maininv}.
If zeros of $\nabla a(D_x)$ are not taken into account, the
weight should change to the one as in estimate \eqref{EQ:main3}.

\section{Smoothing estimates for non-dispersive equations}
\label{SECTION:nondisp}

Canonical transformations and the comparison principle
recalled for convenience 
in Appendix \ref{SECTION:tools} are still important tools
when we discuss
the smoothing estimates for {\em non-dispersive equations}
\[
\left\{
\begin{aligned}
\p{i\partial_t+a(D_x)}\,u(t,x)&=0\quad\text{in $\R_t\times\R^n_x$},\\
u(0,x)&=\varphi(x)\quad\text{in $\R^n_x$},
\end{aligned}
\right.
\]
where real-valued functions $a(\xi)$ fail to
satisfy dispersive assumption (H) or (L) in Section \ref{subsection2.1}.
However, the secondary comparison tools
stated in Section \ref{SECTION:second_comparison} work very effectively
in such situations.
In Corollary \ref{COR:RStype} for example,
even if we lose the dispersiveness assumption at zeros
of $f^\prime$, the estimate is still valid because 
$\sigma$ must vanish at the same points with the order
determined by condition
$|\sigma(\rho)|\leq A |f^\prime(\rho)|^{1/2}$.
The same is true in other comparison results in
Corollary \ref{COR:dimnex}.
In this section, we will treat
the smoothing estimates of non-dispersive equations 
based on these observations.
\medspace
\par
\subsection{Radially symmetric case}
\label{subsection:radial}
The following result states that we still have
estimate \eqref{EQ:maininv} of invariant form
suggested in Section \ref{subsection2.2}
even for non-dispersive equations
in a general setting of the radially symmetric case.
Remarkably enough, it is a straight forward consequence
of the second comparison method of Corollary \ref{COR:RStype},
and in this sense, it is just an equivalent expression of the translation
invariance of the Lebesgue measure (see Section \ref{SECTION:second_comparison}):
\par
\medskip 
\par
\begin{thm}\label{Th:nondisprad}
Suppose $n\geq 1$ and $s>1/2$.
Let $a(\xi)=f(|\xi|)$, where $f\in C^1(\R_+)$ is real-valued.
Assume that $f'$ has only finitely many zeros.
Then we have
\begin{equation}\label{EQ:inv31}
\n{\jp{x}^{-s}|\nabla a(D_x)|^{1/2}
e^{it a(D_x)}\varphi(x)}_{L^2\p{\R_t\times\R^n_x}}
\leq C\n{\varphi}_{L^2\p{\R^n_x}}.
\end{equation}
\end{thm}
\begin{proof}
Noticing $|\nabla a(\xi)|=|f'(|\xi|)|$, use
Corollary \ref{COR:RStype} for $\sigma(\rho)=|f'(\rho)|^{1/2}$
in each interval where $f$ is strictly monotone.
\end{proof}
\par
\medskip 
\par
\begin{example}\label{nondisprad}
As a consequence of Theorem \ref{Th:nondisprad},
we have the estimate of invariant form \eqref{EQ:maininv}
if $a(\xi)$ is a real polynomial of $|\xi|$.
This fact is not a consequence of Theorem \ref{M:H1} or Theorem \ref{M:L4}.
For example, let 
$$a(\xi):=f(|\xi|^2)^2,$$
with $f(\rho)$ being a non-constant real polynomial on $\R$.
The principal part $a_m(\xi)$ of $a(\xi)$
is a power of $|\xi|^2$ multiplied by
a constant, hence it satisfies $\nabla a_m(\xi)\neq0$ ($\xi\neq0$).
If $f(\rho)$ is a homogeneous polynomial,
then $a(\xi)$ satisfies assumption (H) 
and we have estimate \eqref{EQ:maininv} by Theorem \ref{M:H1}.
In the case when $f(\rho)$ is not homogeneous, 
trivially $a(\xi)$ does not satisfy (H).
Furthermore $a(\xi)$ does not always satisfy assumption (L)
either
since
 $$\nabla a(\xi)=4f(|\xi|^2)f'(|\xi|^2)\xi$$ 
vanishes
on the set $|\xi|^2=c$ such that $f(c)=0$ or $f'(c)=0$
as well as at the origin $\xi=0$.
Hence Theorem \ref{M:L4} does not assure the estimate \eqref{EQ:maininv}
for $a(\xi)=(|\xi|^2-1)^2$ (we take $f(\rho)=\rho-1$), for example,
but even in this case, we have the invariant smoothing
estimate \eqref{EQ:inv31} by Theorem \ref{Th:nondisprad}.
\end{example}

\medspace
\par
\subsection{Polynomial case}
\label{subsection:nonradial}
Another remarkable fact is that 
we can obtain invariant estimate \eqref{EQ:maininv}
for all differential equations with real constant coefficients
if we use second comparison method of Corollary \ref{COR:dimnex}
(hence again it is just an equivalent expression of the translation
invariance of the Lebesgue measure):
\par
\medskip 
\par
\begin{thm}\label{polynomial}
Suppose $n\geq 1$ and $s>1/2$.
Let $a(\xi)$ be a real polynomial.
Then we have
\begin{equation}\label{EQ:inv32}
\n{\jp{x}^{-s}|\nabla a(D_x)|^{1/2}
e^{it a(D_x)}\varphi(x)}_{L^2\p{\R_t\times\R^2_x}}
\leq C\n{\varphi}_{L^2\p{\R^2_x}}.
\end{equation}
\end{thm}
\begin{proof}
We can assume that the polynomial $a$ is not a constant since otherwise the estimate is trivial.
Thus, let $m\geq 1$ be the degree of polynomial $a(\xi)$, and we write
it in the form
\[
a(\xi)
=p_0\xi_1^m+p_1(\xi')\xi_1^{m-1}+\cdots+p_{m-1}(\xi')\xi_1+p_m(\xi'),
\]
where $p_k(\xi')$ is a real polynomial in $\xi'=(\xi_2',\ldots,\xi_n')$
of degree $k$ ($k=0,1,\ldots,m$).
The polynomial equation
\[
\partial_1 a(\xi)
=mp_0\xi_1^{m-1}+(m-1)p_1(\xi')\xi_1^{m-2}+\cdots+p_{m-1}(\xi')=0
\]
in $\xi_1$ has at most $m-1$ real roots.
For $k=0,1,\ldots,m-1,$ let $U_k$ be the set of $\xi'\in\R^{n-1}$ for which it
has $k$ distinct real simple roots 
$\lambda_{k,1}(\xi')<\lambda_{k,2}(\xi')\cdots<\lambda_{k,k}(\xi')$, and let
\[
\Omega_{k,l}:=
\b{(\xi_1,\xi')\in\R^n:\xi'\in U_k,\,
\lambda_{k,l}(\xi')<\xi_1<\lambda_{k,l+1}(\xi')}
\]
for $l=0,1,\ldots,k$, regarding $\lambda_{k,0}$ and $\lambda_{k,k+1}$
as $-\infty$ and $\infty$ respectively.
Then we have the decomposition
\[
\Omega=\b{\xi\in\R^n:\partial_1 a(\xi)\neq0}
=\bigcup_{\substack{k=0,1,\ldots,m-1, \\ l=0,1,\ldots,k-1}}\Omega_{k,l},
\]
and $a(\xi)$ is strictly monotone in $\xi_1$ on each $\Omega_{k,l}$.
Hence by taking $\chi$ in Corollary \ref{COR:dimnex}
to be the characteristic functions $\chi_{k,l}$ of the set $\Omega_{k,l}$,
we have the estimate
\[
\n{\jp{x_1}^{-s}\chi_{k,l}(D_x)
\abs{\partial_1 a(D_x)}^{1/2}
e^{it a(D_x)}\varphi(x)}_{L^2(\R_t\times\R_x^n)}
\leq C\n{\varphi}_{L^2(\R_x^n)}.
\]
Taking the sum in $k,l$ and noting the fact that the Lebesgue measure of
the complement of the set $\Omega$ is zero, we have
\[
\n{\jp{x_1}^{-s}
\abs{\partial_1 a(D_x)}^{1/2}
e^{it a(D_x)}\varphi(x)}_{L^2(\R_t\times\R_x^n)}
\leq C\n{\varphi}_{L^2(\R_x^n)}.
\]
Regarding other variables $\xi_j$ ($j=2,\ldots,n$) as the first one,
we also then have
\[
\n{\jp{x_j}^{-s}\abs{\partial_j a(D_x)}^{1/2}
e^{it a(D_x)}\varphi(x)}_{L^2(\R_t\times\R_x^n)}
\leq C\n{\varphi}_{L^2(\R_x^n)},
\]
and combining them all we have
\[
\n{\jp{x}^{-s}
\p{
\abs{\partial_1 a(D_x)}^{1/2}
+\cdots+
\abs{\partial_n a(D_x)}^{1/2}
}
e^{it a(D_x)}\varphi(x)}_{L^2(\R_t\times\R_x^n)}
\leq C\n{\varphi}_{L^2(\R_x^n)}
\]
by trivial inequalities $\jp{x}^{-s}\leq\jp{x_j}^{-s}$.
Substituting $\eta(D_x)\varphi$
for $\varphi$ in the estimate,
where
\[
\eta(\xi)=
|\nabla a(\xi)|^{1/2}
\p{
\abs{\partial_1 a(\xi)}^{1/2}
+\cdots+
\abs{\partial_n a(\xi)}^{1/2}
}
^{-1},
\]
we have
estimate \eqref{EQ:inv32}
if we note the boundedness of $\eta(\xi)$ 
and use Plancherel's theorem.
\end{proof}


\par
\medskip 
\par
\begin{example}\label{nondispersive1}
Some of normal forms listed listed in \eqref {EQ:examples2} in Introduction
satisfy dispersiveness assumptions.
Indeed, $a(\xi_1,\xi_2)=\xi_1^3+\xi_2^3$ and
$\xi_1^3-\xi_1\xi_2^2$ are homogeneous and satisfy
the assumption (H).
The other normal forms however satisfy neither (H) nor (L).
For example, $a(\xi_1,\xi_2)=\xi_1^3$ and
$\xi_1\xi_2^2$ are homogeneous
but $\nabla a(\xi_1,\xi_2)=0$ when $\xi_1=0$ and $\xi_2=0$ respectively.
On the other hand
$a(\xi_1,\xi_2)=\xi_1^3+\xi_2^2$ and $\xi_1\xi_2^2+\xi_1^2$
are not homogeneous
and satisfy $\nabla a(\xi_1,\xi_2)=0$ at the origin.
See Example \ref{ex:isolated-critical} for others.
But even for them we still have invariant smoothing estimate
\eqref{EQ:maininv} with $s>1/2$ by Theorem \ref{polynomial}.
\end{example}
\par
\medskip 
\par

\medspace

\subsection{Hessian at non-dispersive points}
\label{subsection:degenerate}
Now we will present another
approach to treat non-dispersive equations.
Recall that, in \cite[Section 5]{RS4},
the method of canonical transformation effectively works
to reduce smoothing estimates for dispersive equations
listed in Section \ref{subsection2.1}
to some model estimates.
For example, as mentioned in the beginning of
Section \ref{SECTION:second_comparison},
estimate \eqref{EQ:main1} in Theorem \ref{M:H1} is
reduced to model estimates \eqref{model:1} and \eqref{model:2}
in Corollary \ref{Th:typeI}.
We explain here that this strategy works for non-dispersive 
cases as well.

\medskip 
We will however look at the rank of the Hessian $\nabla^2a(\xi)$,
instead of the principal type assumption $\nabla a(\xi)\not=0$.
Assume now that $a=a(\xi)\in C^\infty(\Rn\setminus0)$ is 
real-valued and
positively homogeneous of order two.
It can be noted that 
from Euler's identity we obtain
\begin{equation}\label{id:euler2}
\nabla a(\xi)={}^t\xi\nabla^2 a(\xi)
\end{equation}
since $\nabla a(\xi)$ is homogeneous
of order one (here $\xi$ is viewed as a row).
Then the condition $\rank \nabla^2 a(\xi)=n$ implies
$\nabla a(\xi)\not=0$ ($\xi\neq0$),
and as we have already explained,
we have invariant estimates \eqref{EQ:maininv}, \eqref{EQ:maininv2}
and \eqref{EQ:maininv3}
by Theorem \ref{M:H1} 
in this favourable case.
We will show that in the non-dispersive situation
the rank of $\nabla^2 a(\xi)$ still has a responsibility for
smoothing properties.
We assume 
\begin{equation} \label{a1}
\rank \nabla^2 a(\xi)
\geq k\quad\text{whenever}\quad \nabla a(\xi)=0\quad
(\xi\neq0)
\end{equation}
with some $1\leq k\leq n-1$.
We note that condition \eqref{a1} is invariant under
the canonical transformation
in the following sense:

\begin{lem}\label{l1}
Let $a=\sigma\circ \psi$, with $\psi:U\to\R^n$
satisfying $\det D\psi(\xi)\not=0$ on an open set
$U\subset\R^n$.
Then, for each $\xi\in U$,
$\nabla a(\xi)=0$ if and only if $\nabla \sigma(\psi(\xi))=0$.
Furthermore the ranks of $\nabla^2 a(\xi)$ and $\nabla^2\sigma(\psi(\xi))$
are equal on $\Gamma$ whenever $\xi\in U$ and $\nabla a(\xi)=0$.
\end{lem}

\begin{proof}
Differentiation gives
$\nabla a(\xi)=\nabla\sigma(\psi(\xi))D\psi(\xi)$ and
we have the first assertion.
Another differentiation gives
$\nabla^2 a(\xi)=\nabla^2\sigma (\psi(\xi)) D\psi D\psi$
when $\nabla a(\xi)=0$.
This implies the second assertion.
\end{proof}

To fix the notation, we assume
\begin{equation}\label{a2}
\nabla a(e_n)=0
\quad\text{and}\quad
\rank\nabla^2 a(e_n)=k
\quad(1\leq k\leq n-1),
\end{equation}
where $e_n=(0,\ldots,0,1)$.
Then we have $a(e_n)=0$ by Euler's identity $2a(\xi)=\xi\cdot\nabla a(\xi)$.
We claim that there exists a conic neighbourhood
$\Gamma\subset\Rn\setminus0$ of $e_n$
and a homogeneous $C^\infty$-diffeomorphism $\psi:\Gamma\to\widetilde\Gamma$
(satisfying $\psi(\lambda\xi)=\lambda\psi(\xi)$ for all $\lambda>0$
and $\xi\in\Gamma$)
as appeared in Section \ref{SECTION:canonical}
such that we have the form
\begin{equation}\label{rf}
a(\xi)=(\sigma\circ\psi)(\xi),
\quad
\sigma(\eta)=
c_1\eta_1^2+\cdots+c_k\eta_k^2+r(\eta_{k+1},\ldots,\eta_n),
\end{equation}
where $\eta=(\eta_1,\ldots,\eta_n)$ and $c_j=\pm1$ ($j=1,2,\ldots,k$).
We remark that $r$ must be real-valued and 
positively homogeneous of order two.
\medskip 
We will prove the existence of such $\psi$ that
will satisfy \eqref{rf}.
By \eqref{id:euler2}, \eqref{a2}, and the symmetricity,
all the entries of the matrix $\nabla^2 a (e_n)$ are zero
except for the (perhaps)
non-zero upper left $(n-1)\times(n-1)$ corner matrix.
Moreover, by a linear transformation involving only
the first $(n-1)$ variables of
$\xi=(\xi_1,\ldots,\xi_{n-1},\xi_n)$,
we may assume
$\partial^2a/\partial{\xi_1}^2(e_n)\not=0$.
We remark that \eqref{a2} still holds
under this transformation.
Then, by the Malgrange preparation theorem, we can write
\begin{equation} \label{ra1}
a(\xi)=\pm c(\xi)^2 (\xi_1^2+a_1(\xi^\prime)\xi_1+a_2(\xi^\prime)),
\quad
\xi^\prime=(\xi_2,\ldots,\xi_n).
\end{equation}
locally  in a neighbourhood of $e_n$,
where $c(\xi)>0$ is some strictly positive function, while
function $a_1$ and $a_2$ are smooth and real-valued.
Restricting this expression to the hyperplane $\xi_n=1$, 
and using the homogeneity
\[
a(\xi)=\pm\xi_n^2a(\xi_1/\xi_n,\ldots,\xi_{n-1}/\xi_n,1),
\]
we can extend the expression \eqref{ra1} to a
conic neighbourhood $\Gamma$ of $e_n$, 
so that functions 
$c(\xi), a_1(\xi^\prime)$ and $a_2(\xi^\prime)$ 
are positively homogeneous of orders zero, one, and two, respectively.
Let us define $\psi_0(\xi)=c(\xi)\xi$ and 
$\tau(\eta)=\pm(\eta_1^2+a_1(\eta^\prime)\eta_1+a_2(\eta^\prime))$,
so that $a(\xi)=(\tau\circ\psi_0)(\xi)$,
where we write $\eta=(\eta_1,\eta')$, $\eta'=(\eta_2,\ldots,\eta_n)$.
Furthermore,
let us define
$\psi_1(\xi)=(\xi_1+\frac12 a_1(\xi^\prime),\xi')$,
so that $\tau(\xi)=(\sigma\circ\psi_1)(\xi)$ with
$\sigma(\eta)=\eta_1^2+r(\eta^\prime)$,
where $r(\eta^\prime)=
a_2(\eta^\prime)-\frac14 a_1(\eta^\prime)^2$ is
positively homogeneous of degree two.
Then we have $a=\sigma\circ\psi$, where $\psi=\psi_1\circ\psi_0$,
and thus we have the expression \eqref{rf} with $k=1$.

\medskip 
We note that, by the construction, we have
$\psi(e_n)=c(e_n)(\frac12a_1(e_n'),e_n')$,
where $c(e_n)>0$ and $e_n'=(0,\ldots,0,1)\in\R^{n-1}$.
Then we can see that the function $r(\eta^\prime)$ of
$(n-1)$-variables is
defined on a conic neighbourhood of $e_n'$ in $\R^{n-1}$.
On account of this fact and Lemma \ref{l1},
we can apply the same argument above to $r(\eta^\prime)$,
and repeating the process $k$-times,
we have the expression \eqref{rf}.

\medskip 
To complete the argument,
we check that $\det D\psi_0(\xi)=c(\xi)^n$, which clearly 
implies $\det D\psi(\xi)=c(\xi)^n$, and assures that 
it does not vanish on a sufficiently narrow $\Gamma$.
We observe first that
$D\psi_0(\xi)=c(\xi)I_n+{}^t\xi\nabla c(\xi)$, 
where $I_n$ is the identity $n$ by $n$ matrix. 
We note that if we consider the matrix
$A=(\alpha_i\beta_j)_{i,j=1}={}^t\alpha\cdot  \beta$, where
$\alpha=(\alpha_1,\ldots,\alpha_n)$, $\beta=(\beta_1,\ldots,\beta_n)$,
then $A$ has rank one, so its eigenvalues are $n-1$ zeros and
some $\lambda$. But $\tr A$ is also the sum of the eigenvalues,
hence $\lambda=\tr A$. 
Now, let $\alpha=\xi$, $\beta=\nabla c(\xi)$, and $A={}^t\alpha \cdot\beta$.
Since $c(\xi)$ is homogeneous of order zero, by Euler's identity
we have $\tr A=\xi\cdot\nabla c(\xi)=0$, hence all eigenvalues of
$A$ are zero. It follows now that there is a non-degenerate matrix
$S$ such that $S^{-1} A S$ is strictly upper triangular.
But then $\det D\psi_0(\xi)=\det (c(\xi)I_n + S^{-1}A S)$, where
matrix $c(\xi)I_n + S^{-1}A S$ is upper 
triangular with $n$ copies of
$c(\xi)$ at the diagonal. Hence $\det D\psi_0(\xi)=c(\xi)^n$.

\medskip
On account of the above observations,
we have the following result which states that invariant estimates
\eqref{EQ:maininv}, \eqref{EQ:maininv2} and \eqref{EQ:maininv3}
with $m=2$ still hold for a
class of non-dispersive equations:
\par
\medskip 
\par
\begin{thm}\label{th1}
Let $a\in C^\infty(\Rn\setminus0)$ be real-valued and satisfy
$a(\lambda\xi)=\lambda^2 a(\xi)$ for all $\lambda>0$ and $\xi\neq0$.
Assume that $\rank\nabla^2a(\xi)\geq n-1$ whenever
$\nabla a(\xi)=0$ and $\xi\neq0$.
\begin{itemize}
\item 
Suppose $n\geq2$ and $s>1/2$.
Then we have
\[
\n{\jp{x}^{-s}|\nabla a(D_x)|^{1/2}
e^{it a(D_x)}\varphi(x)}_{L^2\p{\R_t\times\R^n_x}}
\leq C\n{\varphi}_{L^2\p{\R^n_x}}.
\]
\item Suppose $(4-n)/2<\alpha<1/2$, or $(3-n)/2<\alpha<1/2$
in the elliptic case $a(\xi)\neq0$ {\rm (}$\xi\neq0${\rm)}.
Then we have
\[
\n{\abs{x}^{\alpha-1}|\nabla a(D_x)|^{\alpha}e^{ita(D_x)}\varphi(x)}_
{L^2\p{\R_t\times\R^n_x}}
\leq C\n{\varphi}_{L^2\p{\R^n_x}}.
\]
\item Suppose $n>4$, or $n>3$ in the elliptic case $a(\xi)\neq0$
$(\xi\neq0).$
Then we have
\[
\n{\jp{x}^{-1}\jp{\nabla a(D_x)}^{1/2}e^{ita(D_x)}\varphi(x)}_
{L^2\p{\R_t\times\R^n_x}}
\leq C\n{\varphi}_{L^2\p{\R^n_x}}.
\]
\end{itemize}
\end{thm}

\begin{proof}
By microlocalisation and an appropriate rotation,
we may assume $\supp \widehat{\varphi}\subset \Gamma$,
where $\Gamma\subset\Rn\setminus0$ is a sufficiently 
narrow conic neighbourhood
of the direction $e_n=(0,\ldots,0,1)$.
Since everything is alright in the dispersive case 
$\nabla a(e_n)\neq0$ by Theorem \ref{M:H1},
we may assume $\nabla a(e_n)=0$.
We may also assume $n\geq2$ since $\nabla a(e_n)=0$ 
implies $\nabla a(\xi)=0$
for all $\xi\neq0$ in the case $n=1$.
Then we have $\rank\nabla^2a(e_n)\neq n$ by the relation \eqref{id:euler2},
hence $\rank\nabla^2a(e_n)=n-1$ by the assumption
$\rank\nabla^2a(\xi)\geq n-1$.
In the setting \eqref{a2} and \eqref{rf} above, we have
\begin{equation}\label{rank:rem}
\rank\nabla^2\widetilde r(\psi(e_n))=0
\end{equation}
by Lemma \ref{l1}, where
$\widetilde r(\eta)=r(\eta_{k+1},\ldots,\eta_n)$.
Since $k=n-1$ in our case, we can see that $r$ is a function
of one variable and $r^{\prime\prime}$ vanishes
identically by \eqref{rank:rem} and the homogeneity of $r$.
Then $r$ is a polynomial of order one,
but is also positively homogeneous of order two.
Hence we can conclude that $r=0$ identically, and have the relation
\[
a(\xi)=(\sigma\circ\psi)(\xi),
\quad
\sigma(\eta)=
c_1\eta_1^2+\cdots+c_{n-1}\eta_{n-1}^2.
\]
Now, we have the estimates
\begin{align*}
&\n{\jp{x}^{-s} |\nabla \sigma(D_x)|^{1/2}
e^{it \sigma(D_x)}\varphi(x)}_{L^2\p{\R_t\times\R^n_x}}
\leq C\n{\varphi}_{L^2\p{\R^n_x}},
\\
&\n{\abs{x}^{\alpha-1}|\nabla \sigma(D_x)|^{\alpha}
e^{it\sigma(D_x)}\varphi(x)}_{L^2\p{\R_t\times\R^n_x}}
\leq C\n{\varphi}_{L^2\p{\R^n_x}},
\\
&\n{\jp{x}^{-1}\jp{\nabla \sigma(D_x)}^{1/2}
e^{it\sigma(D_x)}\varphi(x)}_{L^2\p{\R_t\times\R^n_x}}
\leq C\n{\varphi}_{L^2\p{\R^n_x}},
\end{align*}
if we use the trivial inequalities $\jp{x}^{-s}\leq\jp{x'}^{-s}$,
$|x|^{\alpha-1}\leq|x'|^{\alpha-1}$ and $\jp{x}^{-1}\leq\jp{x'}^{-1}$,
Theorem \ref{M:H1}
with respect to $x'$, and
the Plancherel theorem in $x_n$,
where $x=(x',x_n)$ and $x'=(x_1,\ldots, x_{n-1})$.
On account of Theorem \ref{Th:caninv}
and the $L^2_{-s}$, $\Dot{L}^2_{\alpha-1}$, $L^2_{-1}$--boundedness
of the operators $I_{\psi,\gamma}$ and $I_{\psi,\gamma}^{-1}$
for $(1/2<)s<n/2$, $-n/2<\alpha-1(<-1/2)$ (see Theorem \ref{Th:L'2k}),
respectively,
we have the conclusion.
\end{proof}
\par
\medskip 
\par
\begin{example}\label{nondispersive3}
The function $a(\xi)=b(\xi)^2$ satisfies the condition in
Theorem \ref{th1}, where $b(\xi)$
is a positively homogeneous function of order one
such that $\nabla b(\xi)\neq0$ ($\xi\neq0$).
Indeed, if $b(\xi)$ is elliptic, then 
$\nabla a(\xi)=2b(\xi)\nabla b(\xi)\neq0$
($\xi\neq0$).
If $b(\xi_0)=0$ at a point $\xi_0\neq0$, then $\nabla a(\xi_0)=0$
and further differentiation immediately yields
$\nabla^2 a(\xi_0)=2^t\nabla b(\xi_0) \nabla b(\xi_0)$,
and clearly we have $\rank\nabla^2 a(\xi_0)\geq1$.
Especially in the case $n=2$, $a(\xi)$ meets the condition in
Theorem \ref{th1}.
As an example, we consider
$
a(\xi)=\frac{\xi_1^2\xi_2^2}{\xi_1^2+\xi_2^2}.
$
Setting $b(\xi)=\xi_1\xi_2/|\xi|$, we clearly have $a(\xi)=b(\xi)^2$
and
$
\nabla b(\xi)
=\p{\frac{\xi_2^3}{|\xi|^3},\frac{\xi_1^3}{|\xi|^3}},
$
hence $\nabla b(\xi)\neq0$ ($\xi\neq0$).
Although $\nabla a(\xi)=0$ on the lines $\xi_1=0$ and $\xi_2=0$,
we have invariant estimates \eqref{EQ:maininv}, \eqref{EQ:maininv2}
and \eqref{EQ:maininv3} in virtue of Theorem \ref{th1}.
This is an illustration of a smoothing estimate for the 
Cauchy problem for an equation like
$$
i\partial_t \Delta u+D_1^2D_2^2 u=0,
$$
which can be reduced to the second order non-dispersive
pseudo-differential equation with symbol $a(\xi)$ above.
Similarly, we have these estimates for more general case
$
a(\xi)=\frac{\xi_1^2\xi_2^2}{\xi_1^2+\xi_2^2}
+\xi^2_3+\cdots+\xi_n^2
$
since we obtain $\rank\nabla^2a(\xi)\geq n-1$
from the observation above.
\end{example}

\medspace

\subsection{Isolated critical points}
\label{subsection:nondegenerate}

Next we consider more general operators $a(\xi)$ of order $m$
which may have some lower order terms.
Then even the most favourable case $\det\nabla^2a(\xi)\not=0$
does not imply the dispersive assumption $\nabla a(\xi)\not=0$.
The method of canonical transformation, 
however, can also allow us to treat this problem by
obtaining localised estimates near points $\xi$ 
where $\nabla a(\xi)=0$.

\medskip 

Assume that $\xi_0$ is a non-degenerate critical point
of $a(\xi)$, that is, that we have $\nabla a(\xi_0)=0$
and $\det \nabla^2 a(\xi_0)\not=0$. Let us microlocalise
around $\xi_0$, so that we only look at what happens around $\xi_0$.
In this case, the order of the symbol $a(\xi)$ does not play any role
and we do not distinguish between the main part and lower
order terms.
Let $\Gamma$ denote a sufficiently small
open bounded neighbourhood
of $\xi_0$ so that $\xi_0$ is the only critical point of
$a(\xi)$ in $\Gamma$. Since $\nabla^2 a(\xi_0)$ is symmetric and
non-degenerate, we may assume
$$\nabla^2 a(\xi_0)={\rm diag}\{\pm 1,\cdots,\pm 1\}$$
by a linear transformation.
By Morse lemma for $a(\xi)$, there exists a diffeomorphism
$\psi:\Gamma\to\widetilde\Gamma\subset\Rn$ 
with an open bounded neighbourhood
of the origin such that
\[
a(\xi)=(\sigma\circ\psi)(\xi),
\quad
\sigma(\eta)=
c_1\eta_1^2+\cdots+c_n\eta_n^2,
\]
where $\eta=(\eta_1,\ldots,\eta_n)$ and $c_j=\pm1$ ($j=1,2,\ldots,n$).
From Theorem \ref{M:H1} applied to operator
$\sigma(D_x)$, we obtain the estimates
\begin{align*}
&\n{\jp{x}^{-s} |\nabla \sigma(D_x)|^{1/2}
e^{it \sigma(D_x)}\varphi(x)}_{L^2\p{\R_t\times\R^n_x}}
\leq C\n{\varphi}_{L^2\p{\R^n_x}}
\qquad(s>1/2),
\\
&\n{\jp{x}^{-1} \jp{\nabla \sigma(D_x)}^{1/2}
e^{it \sigma(D_x)}\varphi(x)}_{L^2\p{\R_t\times\R^n_x}}
\leq C\n{\varphi}_{L^2\p{\R^n_x}}
\qquad(n>2).
\end{align*}
Hence by Theorem \ref{Th:caninv}, together with
the $L^2_{-s}$, $L^2_{-1}$-boundedness of the operators
$I_{\psi,\gamma}$ and $I_{\psi,\gamma}^{-1}$
(which is assured by Theorem \ref{Th:L2k}),
we have these estimates with $\sigma(D_x)$
replaced by $a(D_x)$ assuming $\supp\widehat \va\subset\Gamma$.
On the other hand, we have the same estimates
for general $\va$ by Theorem \ref{M:L4} 
if we assume condition (L).
The above argument, however, assures that the 
following weak assumption
is also sufficient if $a(\xi)$ has finitely many 
critical points and
they are non-degenerate:
\medskip
\begin{equation}\tag{{\bf L$'$}}
\begin{aligned}
&a(\xi)\in C^\infty(\R^n),\qquad 
|\nabla a(\xi)|\geq C\jp{\xi}^{m-1}\quad(\text{for large $\xi\in\R^n$})\quad
\textrm{for some}\; C>0,
\\
&|\partial^\alpha\p{a(\xi)-a_m(\xi)}|\leq C_\alpha\abs{\xi}^{m-1-|\alpha|}
\quad\text{for all multi-indices $\alpha$ and all $|\xi|\gg 1$}.
\end{aligned}
\end{equation}
\medskip

Thus, summarising the above argument, we have established the following result
of invariant estimates \eqref{EQ:maininv} and \eqref{EQ:maininv3}:
\par
\medskip 
\par
\begin{thm}\label{THM:isolated-critical}
Let $a\in C^\infty(\Rn)$ be real-valued and assume that it
has finitely many critical points,
all of which are non-degenerate.
Assume also {\rm{(L$'$)}}.
\begin{itemize}
\item Suppose $n\geq1$, $m\geq1$, and $s>1/2$.
Then we have
\[
\n{\jp{x}^{-s}|\nabla a(D_x)|^{1/2}
e^{it a(D_x)}\varphi(x)}_{L^2\p{\R_t\times\R^n_x}}
\leq C\n{\varphi}_{L^2\p{\R^n_x}}.
\]
\item Suppose $n>2$ and $m\geq1$.
Then we have
\[
\n{\jp{x}^{-1}\jp{\nabla a(D_x)}^{1/2}
e^{it a(D_x)}\varphi(x)}_{L^2\p{\R_t\times\R^n_x}}
\leq C\n{\varphi}_{L^2\p{\R^n_x}}.
\]
\end{itemize}
\end{thm}
\par
\medskip 
\par
\begin{example}
It is easy to see that
$a(\xi)=\xi_1^4+\cdots+\xi_n^4+|\xi|^2$ satisfies the assumption of 
Theorem \ref{THM:isolated-critical}.
We remark that Morii \cite{Mo} also established the first estimate
in Theorem \ref{THM:isolated-critical} under a more restrictive condition
but which also allows this example.
\end{example}

\begin{example}\label{ex:isolated-critical}
Some normal forms listed in \eqref{EQ:examples2}
are also covered by Theorem \ref{THM:isolated-critical}
although they are covered by Theorem \ref{polynomial}.
Indeed $a(\xi_1,\xi_2)={\xi_1^3+\xi_1\xi_2}$ has a unique critical point
at the origin and it is non-degenerate.
It is also easy to see that $a(\xi_1,\xi_2)={\xi_1^3+\xi_2^3+\xi_1\xi_2}$
and ${\xi_1^3-3\xi_1\xi_2^2+\xi_1^2+\xi_2^2}$ have
their critical points at $(\xi_1,\xi_2)=(0,0)$, $(-1/3,-1/3)$ and
$(\xi_1,\xi_2)=(0,0)$, $(1/3,\pm1/\sqrt{3})$, $(-2/3,0)$ respectively,
and all of them are non-degenerate.
\end{example}

\section{Equations with time-dependent coefficients}
\label{SECTION:time-dependent}

Finally we discuss smoothing estimates for equations
with time-dependent coefficients:
\begin{equation}\label{equation-t}
\left\{
\begin{aligned}
\p{i\partial_t+b(t,D_x)}\,u(t,x)&=0\quad\text{in $\R_t\times\R^n_x$},\\
u(0,x)&=\varphi(x)\quad\text{in $\R^n_x$}.
\end{aligned}
\right.
\end{equation}
If the symbol $b(t,\xi)$ is independent of $t$, invariants estimates
\eqref{EQ:maininv}, \eqref{EQ:maininv2} and \eqref{EQ:maininv3} say that
$\nabla_\xi b(t,D_x)$ is responsible for the smoothing property.
The natural question here is what quantity replaces it 
if $b(t,\xi)$ depends on $t$.

\medskip
We can give an answer to this question
if $b(t,\xi)$ is of the product type
\[
b(t,\xi)=c(t)a(\xi),
\]
where we only assume that $c(t)>0$ is a continuous function.
In the case of dispersive and Strichartz estimates for
higher order (in $\partial_t$) equations the
situation may be very delicate and in general depends on
the rates of oscillations of $c(t)$ (see e.g. Reissig \cite{Rei}
for the case of the time-dependent wave equation,
or \cite{RW, CUFRW} for more general equations).

\medskip
For smoothing estimates, we will be able to state a rather
general result in Theorem \ref{Th:time-dependent} below.
The final formulae show that a natural extension of the
invariant estimates of the
previous section still remain valid in this case.
In this special case, the equation \eqref{equation-t} can be 
transformed
to the equation with time-independent coefficients.
In fact, by the assumption for $c(t)$, the function
\[
C(t)=\int^t_0 c(s)\,ds
\] 
is strictly monotone and the inverse $C^{-1}(t)$ exists.
Then the function
\[
v(t,x):=u(C^{-1}(t),x)
\]
satisfies
\[
\partial_tv(t,x)=\frac1{c(C^{-1}(t))}(\partial_tu)(C^{-1}(t),x),
\]
hence $v(t,x)$ solves the equation
\[
\left\{
\begin{aligned}
\p{i\partial_t+a(D_x)}\,v(t,x)=&0,\\
v(0,x)=&\varphi(x),
\end{aligned}
\right.
\]
if $u(t,x)$ is a solution to equation \eqref{equation-t}.
By this argument,
invariant estimates 
for $v(t,x)=e^{ita(D_x)}\varphi(x)$
should imply also estimates for the solution
\[
u(t,x)=v(C(t),x)=e^{i\int^t_0b(s,D_x)\,ds}\varphi(x)
\]
to equation \eqref{equation-t}.
For example, if we notice the relations
\[
\n{v(\cdot,x)}_{L^2}=\n{|c(\cdot)|^{1/2}u(\cdot,x)}_{L^2}
\]
and
\[
c(t)\nabla a(D_x)=\nabla_\xi b(t,D_x),
\]
we obtain the estimate
\begin{equation}\label{EQ:maininv-t}
\n{\jp{x}^{-s}|\nabla_\xi b(t,D_x)|^{1/2}
e^{i\int^t_0b(s,D_x)\,ds}\varphi(x)}_{L^2\p{\R_t\times\R^n_x}}
\leq C\n{\varphi}_{L^2\p{\R^n_x}}
\end{equation}
from the invariant estimate \eqref{EQ:maininv}.
Estimate \eqref{EQ:maininv-t} is a natural extension of the invariant
estimate \eqref{EQ:maininv} to the case of time-dependent coefficients,
which says that
$\nabla_\xi b(t,D_x)$ is still
responsible for the smoothing property.
From this point of view, we may call it an {\it invariant estimate}
too. We can also note that estimate \eqref{EQ:maininv-t} may
be also obtained directly, by formulating an obvious extension of
the comparison principles to the time dependent setting.
We also have similar estimates from the invariant estimates
\eqref{EQ:maininv2} and \eqref{EQ:maininv3}. The same method
of the proof yields the following:
\par
\medskip 
\par
\begin{thm}\label{Th:time-dependent}
Let $[\alpha,\beta]\subset [-\infty,+\infty]$. Assume that
function $c=c(t)$ is continuous on $[\alpha,\beta]$ and that
$c\not=0$ on $(\alpha,\beta)$. Let $u=u(t,x)$ be the solution
of equation \eqref{equation-t} with $b(t,\xi)=c(t)a(\xi)$,
where $a=a(\xi)$ satisfies assumptions of any part of 
Theorems \ref{M:H1}, \ref{M:L4}, \ref{Th:HL}, \ref{Th:nondisprad}, 
\ref{th1} or \ref{THM:isolated-critical}.
Then the smoothing
estimate of the corresponding theorem holds
provided we replace $L^2(\R_t,\Rnx)$ by 
$L^2([\alpha,\beta],\Rnx)$, and insert $|c(t)|^{1/2}$ in the
left hand side norms.
\end{thm}
\par
\medskip 
\par
We note that it is possible that $\alpha=-\infty$ and that
$\beta=+\infty$, in which case by continuity of $c$ at such points we
simply mean that the limits of $c(t)$ exist as
$t\to\alpha+$ and as $t\to\beta-$.

\medskip
To give an example of an estimate 
from Theorem \ref{Th:time-dependent}, 
let us look at the case of the first statement of Theorem \ref{M:H1}.
In that theorem, we suppose that $a(\xi)$ satisfies assumption (H),
and we assume $n\geq 1$, $m>0$, and $s>1/2$. 
Theorem \ref{M:H1} assures that in this case we have the
smoothing estimate \eqref{EQ:main1}, which is 
\[
\n{\jp{x}^{-s}|D_x|^{(m-1)/2}e^{ita(D_x)}\varphi(x)}_
{L^2\p{\R_t\times\R^n_x}}
\leq C\n{\varphi}_{L^2\p{\R^n_x}}.
\]
Theorem \ref{Th:time-dependent} states that solution
$u(t,x)$ of equation \eqref{equation-t} 
satisfies this estimate
provided we replace $L^2(\R_t,\Rnx)$ by 
$L^2([\alpha,\beta],\Rnx)$, and insert $|c(t)|^{1/2}$ in the
left hand side norm. This means that $u$ satisfies
\[
\n{\jp{x}^{-s}|c(t)|^{1/2}|D_x|^{(m-1)/2}u(t,x)}_
{L^2\p{[\alpha,\beta]\times\R^n_x}}
\leq C\n{\varphi}_{L^2\p{\R^n_x}}.
\]
The same is true with statements of any of Theorem 
\ref{M:H1}, \ref{M:L4}, \ref{Th:HL},
\ref{Th:nondisprad}, \ref{th1} or \ref{THM:isolated-critical}.

\appendix
\section{Canonical transformation and comparison principle}
\label{SECTION:tools}

For convenience of the reader in this appendix we
briefly recall two powerful tools introduced in \cite{RS4} for getting smoothing estimates,
that is, the canonical transformation and the comparison principle,
which enable us to induce global smoothing estimates for dispersive equations
rather easily, and formulate several corollaries of these methods to be used in
the analysis of this paper. In particular, Theorem \ref{Th:caninv} explains the
invariance of \eqref{EQ:main-invariant} and similar estimates under canonical
transforms. Also, Corollary \ref{COR:dimnex} is instrumental in treating equations
with polynomial symbols (i.e. differential evolution equations) in Section \ref{subsection:nonradial}.

We remark that all known smoothing estimates from Section \ref{subsection2.1}
were proved in \cite{RS4} by using these two methods.

\medspace

\subsection{Canonical transformation}
\label{SECTION:canonical}

The first tool is the canonical transformation
which transforms the equation
with the operator $a(D_x)$ and
the Cauchy data $\varphi(x)$
to that with $\sigma(D_x)$ and $g(x)$
at the estimate level,
where $a(D_x)$ and $\sigma(D_x)$ are related with each other
as
$a(\xi)=\p{\sigma\circ\psi}(\xi)$.

\medskip 
Let $\Gamma$, $\widetilde{\Gamma}\subset\R^n$ be open sets and
$\psi:\Gamma\to\widetilde{\Gamma}$
be a $C^\infty$-diffeomorphism (we do not assume them to be cones
since we do not require homogeneity of phases).
We always assume that
\begin{equation}\label{infty}
C^{-1}\leq\abs{\det \partial\psi(\xi)}\leq C\quad(\xi\in\Gamma),
\end{equation}
for some $C>0$.
Let
$\gamma\in C^\infty(\Gamma)$ and
$\widetilde{\gamma}=\gamma\circ\psi^{-1}\in C^\infty(\widetilde{\Gamma})$
be cut-off functions
which satisfy $\supp\gamma\subset\Gamma$,
$\supp\widetilde{\gamma}\subset\widetilde{\Gamma}$.
Then we set
\begin{equation}\label{DefI0}
\begin{aligned}
I_{\psi,\gamma} u(x)
&=\FT^{-1}\left[\gamma(\xi)\FT u\p{\psi(\xi)}\right](x)
\\
&=(2\pi)^{-n}\int_{\R^n}\int_{\Gamma}
 e^{i(x\cdot\xi-y\cdot\psi(\xi))}\gamma(\xi)u(y) dyd\xi,
\\
I_{\psi,\gamma}^{-1} u(x)
&=\FT^{-1}\left[\widetilde{\gamma}(\xi)\FT
u\p{\psi^{-1}(\xi)}\right](x)
\\
&=(2\pi)^{-n}\int_{\R^n}
\int_{\widetilde{\Gamma}}
 e^{i(x\cdot\xi-y\cdot\psi^{-1}(\xi))}
 \widetilde{\gamma}(\xi)u(y) dyd\xi.
\end{aligned}
\end{equation}
In the case that $\Gamma$, $\widetilde{\Gamma}\subset\R^n\setminus0$
are open cones,
we may consider the homogeneous $\psi$ and $\gamma$ which satisfy
$\supp\gamma\cap \Sph^{n-1}\subset\Gamma\cap \Sph^{n-1}$ and
$\supp\widetilde{\gamma}\cap 
\Sph^{n-1}\subset\widetilde{\Gamma}\cap \Sph^{n-1}$,
where $\Sph^{n-1}=\b{\xi\in\Rn: |\xi|=1}$.
Then we have the expressions for compositions
\[
I_{\psi,\gamma}\cdot\sigma(D_x)
=\p{\sigma\circ\psi}(D_x)\cdot I_{\psi,\gamma},\quad
I_{\psi,\gamma}^{-1}\cdot \p{\sigma\circ\psi}(D_x)
=\sigma(D_x)\cdot I_{\psi,\gamma}^{-1}
\]
which enable us to relate $a(D_x)$ with $\sigma(D_x)$
when $a(\xi)=\p{\sigma\circ\psi}(\xi)$.

\medskip 
We also introduce the weighted $L^2$-spaces.
For a weight function $w(x)$, let $L^2_{w}(\R^n;w)$ be
the set of measurable functions $f:\Rn\to\C$ 
such that the norm
\[
\n{f}_{L^2(\R^n;w)}
=\p{\int_{\R^n}\abs{w(x) f(x)}^2\,dx}^{1/2}
\]
is finite.
Then we have the following fundamental theorem:
\par
\medskip 
\par
\begin{thm}[{\cite[Theorem 4.1]{RS4}}]\label{Th:reduction}
Assume that the operator $I_{\psi,\gamma}$ defined by \eqref{DefI0}
is $L^2(\R^n;w)$--bounded.
Suppose that we have the estimate
\[
\n{w(x)\rho(D_x)e^{it\sigma(D_x)}g(x)}_{L^2\p{\R_t\times\R^n_x}}
\leq C\n{g}_{L^2\p{\R^n_x}}
\]
for all $g$ such that
$\supp\widehat{g}\subset\supp\widetilde\gamma$.
Assume also that the function
\[
q(\xi)=\frac{\gamma\cdot\zeta}{\rho\circ \psi}(\xi)
\]
is bounded.
Then we have
\[
\n{w(x)\zeta(D_x)e^{ita(D_x)}\varphi(x)}_{L^2\p{\R_t\times\R^n_x}}
\leq C\n{\varphi}_{L^2\p{\R^n_x}}
\]
for all $\varphi$
such that $\supp\widehat{\varphi}\subset\supp\gamma$,
where $a(\xi)=(\sigma\circ\psi)(\xi)$.
\end{thm}
\par
\medskip 
\par
We remark that invariant estimate \eqref{EQ:inv-form}
introduced in Section \ref{subsection2.2}, hence
estimates \eqref{EQ:maininv}--\eqref{EQ:maininv3} are
invariant under canonical transformations by Theorem \ref{Th:reduction}.
More precisely, we have the following theorem: 
\par
\medskip 
\par
\begin{thm}\label{Th:caninv}
Let $\zeta$ be a function on $\R_+$ of the form
$\zeta(\rho)=\rho^\eta$ or $\p{1+\rho^2}^{\eta/2}$
with some $\eta\in\R$.
Assume that the operators $I_{\psi,\gamma}$ and $I_{\psi,\gamma}^{-1}$
defined by \eqref{DefI0} are $L^2(\Rn;w)$--bounded.
Then the following two estimates
\begin{align*}
&\n{w(x)\zeta(|\nabla a(D_x)|)
e^{it a(D_x)}\varphi(x)}_{L^2\p{\R_t\times\R^n_x}}
\leq C\n{\varphi}_{L^2\p{\R^n_x}}
\quad (\supp\widehat\varphi\subset\supp\gamma),
\\
&\n{w(x) \zeta(|\nabla \sigma(D_x)|)
e^{it \sigma(D_x)}\varphi(x)}_{L^2\p{\R_t\times\R^n_x}}
\leq C\n{\varphi}_{L^2\p{\R^n_x}}
\quad (\supp\widehat\varphi\subset\supp\widetilde\gamma)
\end{align*}
are equivalent to each other, where $a=\sigma\circ\psi\in C^1$
on $\supp \gamma$.
\end{thm}
\begin{proof}
Note that $\nabla a(\xi)=\nabla\sigma(\psi(\xi))D\psi(\xi)$
and
$C|\nabla a(\xi)|\leq|\nabla\sigma(\psi(\xi))|\leq C'|\nabla a(\xi)|$
on $\supp\gamma$ with some $C,C'>0$,
which is assured by the assumption \eqref{infty}.
Then the result is obtained from Theorem \ref{Th:reduction}.
\end{proof}
\par
\medskip 
\par
As for the $L^2(\R^n;w)$--boundedness of the operator $I_{\psi,\gamma}$,
we have criteria for some special weight functions.
For $\ka\in\R$, let $L^2_\ka(\R^n)$, $\Dot{L}^2_\ka(\R^n)$ be
the set of measurable functions $f$ such that the norm
\[
\n{f}_{L^2_\ka(\R^n)}
=\p{\int_{\R^n}\abs{\langle x\rangle^\ka f(x)}^2\,dx}^{1/2},
\qquad
\n{f}_{\Dot{L}^2_\ka(\R^n)}
=\p{\int_{\R^n}\abs{|x|^\ka f(x)}^2\,dx}^{1/2}
\]
is finite, respectively.
Then we have the following:
\par
\medskip 
\par
\begin{thm}[{\cite[Theorem 4.2]{RS4}}]\label{Th:L2k}
Suppose $\ka\in\R$.
Assume that all the derivatives of entries of the 
$n\times n$ matrix
$\partial\psi$ and those of $\gamma$ are bounded.
Then the operators $I_{\psi,\gamma}$ and $I^{-1}_{\psi,\gamma}$ 
defined by
\eqref{DefI0} are $L^2_{\ka}(\R^n)$--bounded.
\end{thm}
\begin{thm}[{\cite[Theorem 4.3]{RS4}}]\label{Th:L'2k}
Let $\Gamma$, $\widetilde{\Gamma}\subset\R^n\setminus0$ be open cones.
Suppose $|\ka|< n/2$.
Assume $\psi(\lambda\xi)=\lambda\psi(\xi)$,
$\gamma(\lambda\xi)=\gamma(\xi)$ for all $\lambda>0$ and $\xi\in\Gamma$.
Then the operators $I_{\psi,\gamma}$ and $I^{-1}_{\psi,\gamma}$
defined by \eqref{DefI0} are $L^2_{\ka}(\R^n)$--bounded
and $\Dot{L}^2_{\ka}(\R^n)$--bounded.
\end{thm}
\medspace
\par
\subsection{Comparison principle}
\label{SECTION:comparison}
The second tool is the comparison principle which
relates the smoothing estimate for the solution
$u(t,x)=e^{it f(D_x)}\varphi(x)$ with the operator $f(D_x)$ 
of the smoothing $\sigma(D_x)$
to that for $v(t,x)=e^{it g(D_x)}\varphi(x)$ with
$g(D_x)$ of $\tau(D_x)$:
\par
\medskip 
\par
\begin{thm}[{\cite[Theorem 2.5]{RS4}}]\label{prop:dim1eqmod}
Let $f,g\in C^1(\R_+)$ be real-valued and strictly
monotone on the support of a measurable function $\chi$ on $\R_+$.
Let $\sigma,\tau\in C^0(\R_+)$ be such that, for some $A>0$, we have
\[
\frac{|\sigma(\rho)|}{|f^\prime(\rho)|^{1/2}}
\leq A \frac{|\tau(\rho)|}{|g^\prime(\rho)|^{1/2}}
\]
for all $\rho\in\supp\chi$ satisfying
$f^\prime(\rho)\not=0$ and $g^\prime(\rho)\not=0$.
Then we have
\[
\|\chi(|D_x|)\sigma(|D_x|)e^{it f(|D_x|)}\varphi(x)\|_{L^2(\R_t)}
\leq A
\|\chi(|D_x|)\tau(|D_x|)e^{it g(|D_x|)}\varphi(x)\|_{L^2(\R_t)}
\]
for all $x\in\R^n$.
\end{thm}
\begin{thm}[{\cite[Corollary 2.2]{RS4}}]\label{prop:dimneq}
Let $f,g\in C^1(\R^n)$ be real-valued functions
such that, for almost all $\xi'=(\xi_2,\ldots,\xi_n)\in\R^{n-1}$,
$f(\xi)$ and $g(\xi)$
are strictly monotone in $\xi_1$ 
on the support of a measurable function $\chi$ on $\R^n$.
Let $\sigma,\tau\in C^0(\R^n)$ be such that, for some $A>0$, we have
\[
\frac{|\sigma(\xi)|}{\abs{\partial_{1} f(\xi)}^{1/2}}
\leq A \, \frac{|\tau(\xi)|}{\abs{\partial_{1} g(\xi)}^{1/2}}
\]
for all $\xi\in\supp\chi$ satisfying
$\partial_1 f(\xi)\not=0$ and 
$\partial_1 g(\xi)\not=0$.
Then we have
\begin{multline*}
\n{\chi(D_x)\sigma(D_x)
e^{it f(D_x)}\varphi(x_1,x')}_{L^2(\R_t\times\R_{x'}^{n-1})} 
\\ \leq  A 
\|\chi(D_x)\tau(D_x)e^{it g(D_x)}\varphi(\widetilde x_1,x')\|_
{L^2(\R_t\times\R_{x'}^{n-1})}
\end{multline*}
for all $x_1,\widetilde x_1\in\R$, where $x'=(x_2,\ldots,x_n)\in\R^{n-1}$.
\end{thm}
\par
\medskip 
\par
Let us now repeat here important examples of the use of the comparison
principle discussed in \cite{RS4}.
Applying Theorem \ref{prop:dimneq} with $n=1$ in two directions, we
immediately obtain that for 
$l,m>0$, we have
\begin{equation}\label{prop:dim1ex}
\n{|D_x|^{(m-1)/2}e^{it|D_x|^{m}}
\varphi(x)}_{L^2(\R_t)}=
\sqrt{\frac{l}{m}}
\n{|D_x|^{(l-1)/2}e^{it|D_x|^{l}}
\varphi(x)}_{L^2(\R_t)}
\end{equation}
for every $x\in\R$,
assuming that $\supp\widehat{\varphi}\subset [0,+\infty)$ or
$(-\infty,0]$.
Here we neglect $x'=(x_2,\ldots,x_n)$ 
in a natural way and
just write $x=x_1$ and $D_x=D_1$.
Applying Theorem \ref{prop:dimneq} with $n=2$, we
similarly obtain that for $l,m>0$, we have 
\begin{multline}\label{prop:dim2ex}
\n{|D_y|^{(m-1)/2}e^{itD_x|D_y|^{m-1}}\varphi(x,y)}_{L^2(\R_t\times\R_y)}
\\ =
\n{|D_y|^{(l-1)/2}e^{itD_x|D_y|^{l-1}}\varphi(x,y)}_{L^2(\R_t\times\R_y)}
\end{multline}
for every $x\in\R$.
Here we have used the notation $(x,y)=(x_1,x_2)$,
and $(D_x,D_y)=(D_1,D_2)$.
On the other hand, in the case $n=1$, we have easily
\begin{equation}\label{core}
\n{e^{itD_x}\varphi(x)}_{L^2(\R_t)}
=\n{\varphi}_{L^2\p{\R_{x}}}
\quad \text{for all $x\in\R$},
\end{equation}
which is a straightforward consequence of the fact
$e^{itD_x}\varphi(x)=\varphi(x+t)$
and the translation invariance of the Lebesgue measure.
By using equality \eqref{core},
we can estimate the right hand sides of equalities \eqref{prop:dim1ex}
and \eqref{prop:dim2ex} with $l=1$, and as a result, we have
the following low dimensional pointwise estimates
\begin{align*}
&\n{|D_x|^{(m-1)/2}e^{it|D_x|^m}\varphi(x)}_{L^2(\R_t)}
\leq C\n{\varphi}_{L^2(\R_x)},
\\
&\n{|D_y|^{(m-1)/2}e^{itD_x|D_y|^{m-1}}\varphi(x,y)}_{L^2(\R_t\times\R_y)}
\leq C\n{\varphi}_{L^2\p{\R^2_{x,y}}}
\end{align*}
for all $x\in\R$, from which we  straightforwardly obtain the following result:
\par
\medskip 
\par
\begin{cor}[{\cite[Corollary 3.3]{RS4}}]\label{Th:typeI}
Suppose $n\geq1$, $m>0$, and $s>1/2$.
Then we have
\begin{equation}\label{model:1}
\n{\jp{x_1}^{-s}|D_1|^{(m-1)/2}e^{it|D_1|^m}\varphi(x)}_{L^2(\R_t\times\R^n_x)}
\leq
 C\n{\varphi}_{L^2(\R_x^n)}.
\end{equation}
Suppose $n\geq2$, $m>0$, and $s>1/2$.
Then we have
\begin{equation}\label{model:2}
\n{\jp{x_1}^{-s}|D_n|^{(m-1)/2}e^{itD_1|D_n|^{m-1}}\varphi(x)}
_{L^2(\R_t\times\R^n_x)}
\leq
 C\n{\varphi}_{L^2(\R_x^n)}.
\end{equation}
\end{cor}
\par
\medskip 
\par
We remark that estimate \eqref{model:1} has been already
the invariant estimate \eqref{EQ:maininv}
for the normal form $a(\xi)=\xi_1^m$
(if we replace the weight $\jp{x_1}^{-s}$ by a smaller one $\jp{x}^{-s}$).

\medspace

\subsection{Secondary comparison}
\label{SECTION:second_comparison}
We remark that Corollary \ref{Th:typeI} is just a consequence of 
trivial equality \eqref{core}, and the proof of Theorem \ref{M:H1}
was carried out in \cite{RS4} by reducing estimate \eqref{EQ:main1} to
estimate \eqref{model:1} (elliptic case) or estimate \eqref{model:2}
(non-elliptic case) in Corollary \ref{Th:typeI}
via canonical transformations discussed in Section \ref{SECTION:canonical}.
Let us further compare estimates in Theorem \ref{M:H1}
and Corollary \ref{Th:typeI} by using the comparison principle again
to obtain {\it secondary comparison} results.
In this sense, the results stated below are obtained from just the translation
invariance of the Lebesgue measure via a combination use of the comparison
principle and the canonical transformation.
\par
\medskip 
\par
Now, in notation of Theorem \ref{prop:dim1eqmod}, setting
$\tau(\rho)=\rho^{(m-1)/2}$ and
$g(\rho)=\rho^{m}$, we have
$|\tau(\rho)|/|g'(\rho)|^{1/2}=m^{-1/2}$.
Hence, noticing that $\chi(D_x)$ is $L^2$--bounded for $\chi\in L^\infty$,
we obtain the following result from Theorem \ref{M:H1}
with $a(\xi)=|\xi|^m$:
\par
\medskip 
\par
\begin{cor}[{\cite[Corollary 7.3]{RS4}}]\label{COR:RStype}
Suppose $n\geq1$, $s>1/2$.
Let $\chi\in L^\infty(\R_+)$. 
Let $f\in C^1(\R_+)$ be real-valued and
strictly monotone on $\supp \chi$.
Let $\sigma\in C^0(\R_+)$
be such that for some $A>0$ we have
\[
|\sigma(\rho)|\leq A |f^\prime(\rho)|^{1/2}
\]
for all $\rho\in\supp\chi$.
Then we have
\[
\n{\jp{x}^{-s}\chi(|D_x|)\sigma(|D_x|)
e^{itf(|D_x|)}\varphi(x)}_\L2tx
\leq C\n{\varphi}_\Lx.
\]
\end{cor}
\par
\medskip 
\par
Similarly, in notation of Theorem \ref{prop:dimneq}, setting
$\tau(\xi)=|\xi_1|^{(m-1)/2}$ and
$g(\xi)=|\xi_1|^{m}$, we have
$|\tau(\xi)|/|\partial g/\partial \xi_1 (\xi)|^{1/2}=m^{-1/2}$.
Then we obtain the following result from estimate \eqref{model:1} of
Corollary \ref{Th:typeI}:
\par
\medskip 
\par
\begin{cor}\label{COR:dimnex}
Suppose $n\geq 1$ and $s>1/2$.
Let $\chi\in L^\infty(\R^n)$. 
Let $f\in C^1(\R^n)$ be a real-valued function such that,
for almost all $\xi'=(\xi_2,\ldots,\xi_n)\in\R^{n-1}$ ,
$f(\xi)$ is strictly monotone in $\xi_1$
on $\supp\chi$.
Let $\sigma\in C^0(\R^n)$ be such that
for some $A>0$ we have
\[
|\sigma(\xi)|
\leq A \left|\partial_1 f(\xi)
\right|^{1/2}
\]
for all $\xi\in\supp\chi$.
Then we have
\[
\n{\jp{x_1}^{-s}\chi(D_x)\sigma(D_x)
e^{it f(D_x)}\varphi(x)}_{L^2(\R_t\times\R_x^n)}\leq C
\n{\varphi}_\Lx.
\]
\end{cor}



\end{document}